\documentclass[a4paper,12pt,reqno]{amsart}
\usepackage{amsmath,amsthm,amssymb}
\usepackage{graphicx, color}
\usepackage[numbers,sort&compress]{natbib}
\nonstopmode \numberwithin{equation}{section}

\newtheorem{theorem}{Theorem}[section]

 \newtheorem{corollary}{Corollary}[section]

\newtheorem{remark}{Remark}[section]

\allowdisplaybreaks

\allowdisplaybreaks

\begin{document}

\title[Hurwitz Lerch Zeta function]{Further Extension of the Generalized Hurwitz-Lerch Zeta Function of Two Variables}

\author[K.S. Nisar]{Kottakkaran Sooppy  Nisar*}

\address{Kottakkaran Sooppy  Nisar:    Department of Mathematics, College of Arts and Science-Wadi Aldawaser, 11991,
Prince Sattam bin Abdulaziz University, Alkharj, Saudi Arabia}
\email{n.sooppy@psau.edu.sa}

\subjclass[2010]{11M06, 11M35, 33B15, 33C60}
\keywords{gamma function; beta function; hypergeometric function; generalized Hurwitz-Lerch zeta function.}

\thanks{*Corresponding author}

\begin{abstract}
The main aim of this paper is to provide a new generalization of Hurwitz-Lerch Zeta function of two variables. We also investigate several interesting properties such as integral representations, summation formula, and a connection with the generalized hypergeometric function. To strengthen the main results we also consider some important special cases.
\end{abstract}

\maketitle

\section{Introduction}\label{sec-1}
The generalized hypergeometric function $F(-)$  \cite{Erdelyi} defined by   
\begin{align}\label{ghyp}
F\left(\beta_1,\cdots, \beta_p;\delta_1,\cdots,\delta_q;z\right)
&=_pF_q\left(\beta_1,\cdots, \beta_p;\delta_1,\cdots,\delta_q;z\right)\notag\\
&=\sum_{n=0}^{\infty}\frac{\prod_{i=1}^{p}\left(\beta_i\right)_{n}z^{n}}{\prod_{j=1}^{q}\left(\beta_j\right)_{n}n!},
\end{align}
where $p,q\in \mathbb{Z}^{+};b_{j}\neq 0,-1,-2,\cdots$.

The Appell hypergeometric function $F_{1}$ of two variables \cite{SK1985} is defined by 
\begin{eqnarray}
\aligned
F_{1}[a,b,b';c;z,t]=\sum_{k,l=0}^{\infty}\frac{(a)_{k+l}(b)_k(b')_{l}}{(c)_{k+l}}\frac{z^{k}}{k!}\frac{t^{l}}{l!},\\
=\sum_{k=0}^{\infty}\frac{(a)_{k}(b)_k}{(c)_{k}}~
_2F_1 \left[ \aligned a+k,\,b' &;\\
                       c+k &; \endaligned
                  \,\, t \right]\frac{z^{k}}{k!},\\
									\left(max\left\{\Re(z), \Re(t)\right\}\leq 1\quad \text{and}\quad \Re(a)>0\right).
\endaligned
\end{eqnarray}\label{Phi1}

The confluent forms of Humbert functions are \cite{SK1985}:
\begin{equation}
\Phi_{1}[a,b;c;z,t]=\sum_{k,l=0}^{\infty}\frac{(a)_{k+l}(b)_k}{(c)_{k+l}}\frac{z^{k}}{k!}\frac{t^{l}}{l!}, \quad \left(\left|z\right|< 1, \left|t\right|< \infty\right),
\end{equation}
\begin{equation}\label{Phi2}
\Phi_{2}[b,b';c;z,t]=\sum_{k,l=0}^{\infty}\frac{(b)_{k}(b')_l}{(c)_{k+l}}\frac{z^{k}}{k!}\frac{t^{l}}{l!}, \quad \left(\left|z\right|< \infty, \left|t\right|< \infty\right),
\end{equation}
and
\begin{equation}\label{Phi3}
\Phi_{3}[b;c;z,t]=\sum_{k,l=0}^{\infty}\frac{(b)_k}{(c)_{k+l}}\frac{z^{k}}{k!}\frac{t^{l}}{l!},\quad \left(\left|z\right|< \infty, \left|t\right|< \infty\right).
\end{equation}

The Appell's type generalized functions $M_{i}$ by considering product of two $_3F_2$ functions is given in \cite{MAK-GSA}. From these expansions, we recall one of the generalized Appell's type functions of two variables $M_{4}$ and is defined by
\begin{equation}\label{M4}
\aligned
M_{4}\left(\mu,\eta,\eta',\delta,\delta';\nu,\xi,\xi';x,y\right)
=\sum_{k=0}^{\infty}\sum_{l=0}^{\infty}\frac{(\mu)_{k+l}(\eta)_{k}(\eta')_{l}(\delta)_{k}(\delta')_{l}}{(\nu)_{k+l}(\xi)_{k}(\xi')_{l}}
\frac{x^{k}}{k!}\frac{y^{l}}{l!}.
\endaligned
\end{equation}

If we set $\mu=\nu, \delta=\xi, \delta'=\xi'$ in \eqref{M4} then
\begin{equation}\label{M4b}
M_{4}[\mu,\eta,\eta',\delta, \delta';\mu,\delta,\delta';x,y]=\left(1-x\right)^{-\eta}\left(1-y\right)^{-\eta'}.
\end{equation}

The Hurwitz-Lerch Zeta function $\Phi(z,s,a)$ is defined by (see \cite{HC2001, HC2012}):
\begin{equation}\label{Zeta}
\Phi(z,s,a)=\sum_{k=0}^{\infty}\frac{z^{k}}{(k+a)^{s}},
\end{equation}
$$\quad \left(a\in\mathbb{C}\setminus \mathbb{Z}_{0}^{-};s\in \mathbb{C}\right) \text {when} \quad \left|z\right|<1; \Re(s)>1 \text{when} \quad\left|z\right|=1.$$

{For more details about the properties and particular cases found in \cite{Erdelyi, HC2001, HC2012}. Various type of  generalizations, extensions, and properties of the Hurwitz-Lerch Zeta function can be found in \cite{CZ2001, CJS2008, JPS2011, LS2004, SJPS2011, SLR2013, SSPS2011, HMS2014}.}

Recently, Pathan and Daman \cite{PD2012} give another generalization of the form
\begin{equation}\label{PD-GHZ}
\aligned
&\Phi_{\alpha,\beta;\gamma, \lambda, \mu; \nu}\left(z,t,s,a\right):=\sum_{k,l=0}^{\infty}\frac{(\alpha)_{k}(\beta)_{k}(\lambda)_{l}(\mu)_{l}}{(\gamma)_{k}(\nu)_{l}k!l!}\frac{z^{k}t^{l}}{(k+l+a)^{s}}, \\
&\quad \gamma, \nu, a \neq \left\{0,-1,-2,\cdots,\right\}, s\in\mathbb{C};\\
&\quad\quad\quad\Re\left(s+\gamma+\nu-\alpha-\beta-\mu-\lambda\right)>0 \quad \text {when} \quad \left|z\right|=1 \quad \text{and} \quad\left|t\right|=1.
\endaligned
\end{equation}

Very recently, Choi and Parmar \cite{CP2017} introduced two variable generalization by 
\begin{equation}\label{EGHLZ}
\aligned
&\Phi_{\mu,\eta,\eta';\nu}\left(z,t,s,a\right):=\sum_{k,l=0}^{\infty}\frac{(\mu)_{k+l}(\eta)_{k}(\eta')_{l}}{(\nu)_{k+l}k!l!}
\frac{z^{k}t^{l}}{(k+l+a)^{s}},  \\
&\quad\left(\mu, \eta,\eta'\in\mathbb{C};\quad a, \nu\in\mathbb{C}\setminus \mathbb{Z}_{0}^{-};\quad s,z,t \in \mathbb{C} \quad \text{when} \quad \left|z\right|<1\quad \text{and} \quad \left|t\right|<1;\right. \\
&\left.\quad\quad\quad \text{and} \quad\Re\left(s+\nu-\mu-\eta-\eta'\right)>0 \quad \text {when} \quad \left|z\right|=1 \quad \text{and} \quad\left|t\right|=1\right) .
\endaligned
\end{equation}
 
{In this paper, we further extended the Hurwitz-Lerch Zeta function of two variables and is defined~by}
\begin{equation}
\begin{array}{ccccc}\label{FEGHLZ}
&\Phi_{\mu,\eta,\eta',\delta,\delta';\nu,\xi,\xi'}\left(z,t,s,a\right):=\sum_{k,l=0}^{\infty}\frac{(\mu)_{k+l}(\eta)_{k}(\eta')_{l}(\delta)_{k}(\delta')_{l}}{(\nu)_{k+l}(\xi)_{k}(\xi')_{l}k!l!}
\frac{z^{k}t^{l}}{(k+l+a)^{s}},\\
&\left(\mu, \eta,\eta', \delta, \delta' \in\mathbb{C}; a, \nu, \xi, \xi' \in\mathbb{C}\setminus \mathbb{Z}_{0}^{-};\, s,z,t \in \mathbb{C} \quad \text{when} \quad \left|z\right|<1 \quad \text{and} \quad \left|t\right|<1;\right.\\
&\left.\text{and}\quad\Re\left(s+\nu+\xi+\xi'-\mu-\eta-\eta'-\delta-\delta'\right)>0 \quad \text {when}  \left|z\right|=1
 \text{and} \left|t\right|=1\right).
\end{array}\end{equation}

{\bf {Special cases}:}   

Case 1. If $\delta=\xi, \delta'=\xi'$, then \eqref{FEGHLZ} reduces to (3) of \cite{CP2017} which is given in \eqref{EGHLZ}.

Case 2. If $\mu=\nu$ and $\delta=\xi, \delta'=\xi'$ in \eqref{FEGHLZ}, then we get the generalized Hurwitz-Lerch Zeta function of \cite{PD2012}:
\begin{equation}\label{HZ-Pathan}
\aligned
&\Phi_{\eta,\eta'}\left(z,t,s,a\right):=\sum_{k,l=0}^{\infty}\frac{(\eta)_{k}(\eta')_{l}}{k!l!}
\frac{z^{k}t^{l}}{(k+l+a)^{s}},  \\
&\quad\left(\eta, \eta'\in\mathbb{C};\quad a\in\mathbb{C}\setminus  \mathbb{Z}_{0}^{-};\quad s \in \mathbb{C} \quad \text{when} \quad \left|z\right|<1 \quad \text{and} \quad \left|t\right|<1;\right. \\
&\left.\quad\quad\quad \text{and} \quad\Re\left(s-\eta-\eta'\right)>0 \quad \text {when} \quad \left|z\right|=1 \quad \text{and} \quad\left|t\right|=1\right).
\endaligned
\end{equation}

The limiting cases of \eqref{FEGHLZ} are as follows:\\

Case 3.
If $\eta'\rightarrow \infty$ then we have
\begin{equation}\label{c3}
\aligned
&\Phi_{\mu,\eta,\delta,\delta';\nu,\xi,\xi'}\left(z,t,s,a\right):=\lim_{\eta'\rightarrow \infty}\left\{\Phi_{\mu,\eta,\eta',\delta,\delta';\nu,\xi,\xi'}\left(z,t/\eta',s,a\right)\right\}\\
&=\sum_{k,l=0}^{\infty}\frac{(\mu)_{k+l}(\eta)_{k}(\delta)_{k}(\delta')_{l}}{(\nu)_{k+l}(\xi)_{k}(\xi')_{l}k!l!}
\frac{z^{k}t^{l}}{(k+l+a)^{s}},  \\
&\quad\left(\mu, \eta, \delta, \delta' \in\mathbb{C}; a, \nu, \xi, \xi' \in\mathbb{C}\setminus  \mathbb{Z}_{0}^{-};\, s,z,t \in \mathbb{C} \quad \text{when} \quad \left|z\right|<1 \quad \text{and} \quad \left|t\right|<1;\right. \\
&\left.\text{and}\quad\Re\left(s+\nu+\xi+\xi'-\mu-\eta-\delta-\delta'\right)>0 \quad \text {when} \quad \left|z\right|=1  \text{and}\left|t\right|=1\right).
\endaligned
\end{equation}

Case 4.
If $\mu\rightarrow \infty$ then we have
\begin{equation}\label{c4}
\aligned
&\Phi_{\eta,\eta',\delta,\delta';\nu,\xi,\xi'}\left(z,t,s,a\right):=\lim_{\mu\rightarrow \infty}\left\{\Phi_{\mu,\eta,\eta',\delta,\delta';\nu,\xi,\xi'}\left(z/\mu,t/\mu,s,a\right)\right\}\\
&=\sum_{k,l=0}^{\infty}\frac{(\eta)_{k}(\eta')_{l}(\delta)_{k}(\delta')_{l}}{(\nu)_{k+l}(\xi)_{k}(\xi')_{l}k!l!}
\frac{z^{k}t^{l}}{(k+l+a)^{s}},  \\
&\quad\left(\eta, \eta', \delta, \delta' \in\mathbb{C}; a, \nu, \xi, \xi' \in\mathbb{C}\setminus  \mathbb{Z}_{0}^{-};\, s,z,t \in \mathbb{C} \quad \text{when} \quad \left|z\right|<1 \quad \text{and} \quad \left|t\right|<1;\right. \\
&\left.\text{and} \quad\Re\left(s+\nu+\xi+\xi'-\eta-\delta-\delta'\right)>0 \quad \text {when} \quad \left|z\right|=1 \quad \text{and} \quad\left|t\right|=1\right).
\endaligned
\end{equation}

Case 5.
If $min{(\left|\mu\right|,\left|\eta'\right|)} \rightarrow \infty$ then we have
\begin{equation}\label{c5}
\aligned
&\Phi_{\eta,\delta,\delta';\nu,\xi,\xi'}\left(z,t,s,a\right):=\lim_{min{(\left|\mu\right|,\left|\eta'\right|)} \rightarrow \infty}\left\{\Phi_{\mu,\eta,\eta',\delta,\delta';\nu,\xi,\xi'}\left(\frac{z}{\mu},\frac{t}{(\mu \eta')},s,a\right)\right\}\\
&=\sum_{k,l=0}^{\infty}\frac{(\eta)_{k}(\delta)_{k}(\delta')_{l}}{(\nu)_{k+l}(\xi)_{k}(\xi')_{l}k!l!}
\frac{z^{k}t^{l}}{(k+l+a)^{s}},  \\
&\quad\left(\eta, \delta, \delta' \in\mathbb{C}; a, \nu, \xi, \xi' \in\mathbb{C}\setminus  \mathbb{Z}_{0}^{-};\, s,z,t \in \mathbb{C} \quad \text{when} \quad \left|z\right|<1 \quad \text{and} \quad \left|t\right|<1;\right. \\
&\left.\quad\quad \text{and} \quad\Re\left(s+\nu+\xi+\xi'-\eta-\delta-\delta'\right)>0 \quad \text {when} \quad \left|z\right|=1 \quad \text{and} \quad\left|t\right|=1\right).
\endaligned
\end{equation}

\section{Integral Representations}

\begin{theorem}\label{Th-1}
The following integral representation of \eqref{FEGHLZ} holds true:
\begin{align}
&\Phi_{\mu,\eta,\eta',\delta,\delta';\nu,\xi,\xi'}\left(z,t,s,a\right)\notag\\
&=\frac{1}{\Gamma(s)}\int_{0}^{\infty}x^{s-1}e^{-ax}M_{4}\left(\mu,\eta,\eta',\delta,\delta';\nu,\xi,\xi';ze^{-x},te^{-x}\right)dx,\\
&\left(\min\left\{\Re(s),\Re(a)\right\}>0\quad \text{when} \left|z\right|\leq 1 (z\neq1), \left|t\right|\leq 1 (t\neq 1),\right.\notag\\
 &\left.\Re(s)>1, \text{when}~ z=1, t=1\right)\notag.
\end{align}
\end{theorem}

\begin{proof}
Using the following Eulerian integral
\begin{equation}\label{EI}
\frac{1}{\left(k+l+a\right)^{s}}:=\frac{1}{\Gamma(s)}\int_{0}^{\infty}t^{s-1}e^{-(k+l+a)t}dt \quad \left(\min{\Re(s),\Re(a)}>0, k,l \in \mathbb{N}_{0}\right)
\end{equation}
in \eqref{FEGHLZ}, we get
\begin{eqnarray*}
&&\Phi_{\mu,\eta,\eta',\delta,\delta';\nu,\xi,\xi'}\left(z,t,s,a\right)\\
&=&\sum_{k,l=0}^{\infty}\frac{(\mu)_{k+l}(\eta)_{k}(\eta')_{l}(\delta)_{k}(\delta')_{l}}{(\nu)_{k+l}(\xi)_{k}(\xi')_{l}}
\frac{z^{k}t^{l}}{k!l!}\frac{1}{\Gamma(s)}\int_{0}^{\infty}x^{s-1}e^{-(k+l+a)x}dx.  \\
\end{eqnarray*}

Interchanging the order of integration and summation, which is verified by
uniform convergence of the involved series under the given conditions, we have
\begin{eqnarray}\label{theq1}
&&\Phi_{\mu,\eta,\eta',\delta,\delta';\nu,\xi,\xi'}\left(z,t,s,a\right)\notag\\
&=&\frac{1}{\Gamma(s)}\int_{0}^{\infty}x^{s-1}e^{-ax}\sum_{k,l=0}^{\infty}\frac{(\mu)_{k+l}(\eta)_{k}(\eta')_{l}(\delta)_{k}(\delta')_{l}}{(\nu)_{k+l}(\xi)_{k}(\xi')_{l}}
\frac{z^{k}t^{l}}{k!l!}\left(e^{-t}\right)^{k}\left(e^{-t}\right)^{l}dx.
\end{eqnarray}

In view of \eqref{M4}, we arrived the desired result.
\end{proof}

Similarly, if we use \eqref{EI} in the limiting cases \eqref{c3}, \eqref{c4} and \eqref{c5} then we obtain the following~corollaries:
\begin{corollary}\label{cor1}
The following integral representations for
$\Phi_{\mu,\eta,\delta,\delta';\nu,\xi,\xi'}\left(z,t,s,a\right), \Phi_{\eta,\eta',\delta,\delta';\nu,\xi,\xi'}\left(z,t,s,a\right)$ and $\Phi_{\eta,\delta,\delta';\nu,\xi,\xi'}\left(z,t,s,a\right)$ in \eqref{c3}, \eqref{c4} and \eqref{c5} holds true when $\delta=\xi, \delta'=\xi'$ :
\begin{equation}
\Phi_{\mu,\eta;\nu}\left(z,t,s,a\right)\\
=\frac{1}{\Gamma(s)}\int_{0}^{\infty}x^{s-1}e^{-ax}\Phi_{1}\left(\mu,\eta;\nu;ze^{-x},te^{-x}\right)dx,\\ 
\end{equation}
which is (14) of \cite{CP2017}.
\begin{equation}
\Phi_{\eta,\eta';\nu}\left(z,t,s,a\right)\\
=\frac{1}{\Gamma(s)}\int_{0}^{\infty}x^{s-1}e^{-ax}\Phi_{2}\left(\eta,\eta';\nu;ze^{-x},te^{-x}\right)dx,\\
\end{equation}
which is (15) of \cite{CP2017} 
and
\begin{equation}
\Phi_{\eta;\nu}\left(z,t,s,a\right)
=\frac{1}{\Gamma(s)}\int_{0}^{\infty}x^{s-1}e^{-ax}\Phi_{3}\left(\eta;\nu;ze^{-x},te^{-x}\right)dx.
\end{equation}
$\left(\min\left\{\Re(s),\Re(a)\right\}>0 \text{when} \left|z\right|\leq 1 (z\neq1), \left|t\right|\leq 1 (t\neq 1), \right.\\$
$\left.\quad\quad\quad\quad\quad\quad\quad\Re(s)>1, \text{when}~ z=1, t=1\right)$,
which is (16) of \cite{CP2017}.
\end{corollary}

\begin{corollary}\label{cor2}
In view of \eqref{M4b}, we have
\begin{align}\label{18cp}
&\Phi_{\mu,\eta, \eta';\mu}\left(z,t,s,a\right)\notag\\
&=\frac{1}{\Gamma(s)}\int_{0}^{\infty}\frac{x^{s-1}e^{-ax}}{\left(1-ze^{-x}\right)^{\eta}\left(1-te^{-x}\right)^{-\eta'}}dx,
\end{align}
$\left(\min\Re(s)>0,\Re(a)>0 \quad \text{when} \left|z\right|\leq 1 (z\neq1), \left|t\right|\leq 1 (t\neq 1), \right.$\\
$\left.\quad\quad\quad\quad\quad\quad\quad\quad\quad\quad\quad\quad\quad\quad\quad\quad\Re(s)>1, \text{when}~ z=1, t=1\right).$
\end{corollary}

\begin{remark}\label{rem1}
If we take $t=0$ in \eqref{18cp}, then it gives (19) of \cite{CP2017} and by setting $t=0, \eta=1$ then \eqref{18cp} reduces to  (20) of \cite{CP2017}
\end{remark}

\begin{theorem}\label{Th2}
Each of the following integrals for $\Phi_{\mu,\eta,\eta',\delta,\delta';\nu,\xi,\xi'}\left(z,t,s,a\right)$ holds true
\begin{eqnarray}
\aligned
&\Phi_{\mu,\eta,\eta',\delta,\delta';\nu,\xi,\xi'}\left(z,t,s,a\right)\\
&=\frac{\Gamma(\nu)}{\Gamma(\mu)\Gamma(\nu-\mu)}\int_{0}^{\infty}\frac{y^{\mu-1}}{(1+y)^{\nu}}\Phi_{\eta,\eta',\xi,\delta;\delta',\xi'}\left(\frac{zy}{1+y},\frac{ty}{1+y},s,a\right),
\endaligned
\end{eqnarray}
and
\begin{eqnarray}
\aligned
&\Phi_{\mu,\eta,\eta',\delta,\delta';\nu,\xi,\xi'}\left(z,t,s,a\right)\\
&=\frac{\Gamma(\nu)}{\Gamma(s)\Gamma(\mu)\Gamma(\nu-\mu)}\int_{0}^{\infty}\int_{0}^{\infty}\frac{x^{s-1}e^{-ax}y^{\mu-1}}{(1+y)^{\nu}}\\
&\times~\sum_{k=0}^{\infty}\frac{(\eta)_k(\delta)_k}{k!(\xi)_k} \left(\frac{zye^{-x}}{1+y}\right)^{k}\sum_{l=0}^{\infty}\frac{(\eta')_l(\delta')_k}{l!(\xi')_k} \left(\frac{tye^{-x}}{1+y}\right)^{l}dxdy.
\endaligned
\end{eqnarray}
\end{theorem}

\begin{proof}
Setting $a=\mu+k+l, b=\nu+k+l$ in the Eulerian beta function formula,
\begin{equation}
B\left(a,b-a\right)=\frac{\Gamma(a)\Gamma(b-a)}{\Gamma(b)}=\int_{0}^{\infty}\frac{y^{a-1}}{(1+y)^{b}}dy, \Re(b)>\Re(a)>0,
\end{equation}
gives
\begin{equation}\label{pf2-1}
\frac{\Gamma(\mu+k+l)\Gamma(\nu-\mu)}{\Gamma(\nu+k+l)}=\int_{0}^{\infty}\frac{y^{\mu+k+l-1}}{(1+y)^{\nu+k+l}}dy,
\end{equation}
\begin{eqnarray*} 
\Rightarrow \frac{(\mu)_{k+l}\Gamma(\mu)\Gamma(\nu-\mu)}{(\nu)_{k+l}\Gamma(\nu)}=\int_{0}^{\infty}\frac{y^{\mu+k+l-1}}{(1+y)^{\nu+k+l}}dy,\\
\quad \Re(\nu)>\Re(\mu)>0, k,l \in \mathbb{N}.
\end{eqnarray*}
\begin{equation} \label{Pf2-2}
\Rightarrow \frac{(\mu)_{k+l}}{(\nu)_{k+l}}=\frac{\Gamma(\nu)}{\Gamma(\mu)\Gamma(\nu-\mu)}\int_{0}^{\infty}\frac{y^{\mu+k+l-1}}{(1+y)^{\nu+k+l}}dy.
\end{equation}

Now substituting \eqref{Pf2-2} in \eqref{FEGHLZ}, we get
\begin{eqnarray}
\aligned
&\Phi_{\mu,\eta,\eta',\delta,\delta';\nu,\xi,\xi'}\left(z,t,s,a\right)\\
&=\sum_{k=0}^{\infty}\frac{\Gamma(\nu)}{\Gamma(\mu)\Gamma(\nu-\mu)}\int_{0}^{\infty}\frac{y^{\mu+k+l-1}}{(1+y)^{\nu+k+l}}
\frac{(\eta)_{k}(\eta')_{l}(\delta)_{k}(\delta')_{l}}{(\xi)_{k}(\xi')_{l}k!l!}
\frac{z^{k}t^{l}}{(k+l+a)^{s}}dy
\endaligned
\end{eqnarray}
interchanging integration and summation gives
\begin{eqnarray}
\aligned
&\Phi_{\mu,\eta,\eta',\delta,\delta';\nu,\xi,\xi'}\left(z,t,s,a\right)\\
&=\frac{\Gamma(\nu)}{\Gamma(\mu)\Gamma(\nu-\mu)}\int_{0}^{\infty}\frac{y^{\mu-1}}{(1+y)^{\nu}}\sum_{k=0}^{\infty}
\frac{(\eta)_{k}(\eta')_{l}(\delta)_{k}(\delta')_{l}}{(\xi)_{k}(\xi')_{l}k!l!}
\left(\frac{zy}{1+y}\right)^{k}\left(\frac{ty}{1+y}\right)^{l}\frac{1}{\left(k+l+a\right)^{s}}dy.
\endaligned
\end{eqnarray}

In view of \eqref{FEGHLZ} and \eqref{PD-GHZ} we arrived the desired result.

Now, we prove the second integral.
From \eqref{theq1}, $\Phi_{\mu,\eta,\eta',\delta,\delta';\nu,\xi,\xi'}\left(z,t,s,a\right)$ can be written as
\begin{eqnarray*}
&&\Phi_{\mu,\eta,\eta',\delta,\delta';\nu,\xi,\xi'}\left(z,t,s,a\right)\\
&=&\frac{1}{\Gamma(s)}\int_{0}^{\infty}x^{s-1}e^{-ax}\sum_{k,l=0}^{\infty}
\frac{(\mu)_{k+l}(\eta)_{k}(\eta')_{l}(\delta)_{k}(\delta')_{l}}{(\nu)_{k+l}(\xi)_{k}(\xi')_{l}}
\frac{\left(ze^{-x}\right)^{k}}{k!}\frac{\left(te^{-x}\right)^{l}}{l!}dx,
\end{eqnarray*}

Now using \eqref{Pf2-2}, we get
\begin{align*}
\Phi_{\mu,\eta,\eta',\delta,\delta';\nu,\xi,\xi'}\left(z,t,s,a\right)&=\frac{1}{\Gamma(s)}\int_{0}^{\infty}x^{s-1}e^{-ax}\sum_{k,l=0}^{\infty}\frac{\Gamma(\nu)}{\Gamma(\mu)\Gamma(\nu-\mu)}\int_{0}^{\infty}\frac{y^{\mu+k+l-1}}{(1+y)^{\nu+k+l}}dy\\
&\times \frac{(\eta)_{k}(\eta')_{l}(\delta)_{k}(\delta')_{l}}{(\xi)_{k}(\xi')_{l}}\frac{\left(ze^{-x}\right)^{k}}{k!}\frac{\left(te^{-x}\right)^{l}}{l!}dx,\\
&= \frac{\Gamma(\nu)}{\Gamma(s)\Gamma(\mu)\Gamma(\nu-\mu)}\int_{0}^{\infty}\int_{0}^{\infty}\frac{x^{s-1}e^{-ax}y^{\mu-1}}{(1+y)^{\nu}}\\
&\times~\sum_{k=0}^{\infty}\frac{(\eta)_k(\delta)_k}{(\xi)_{k}k!} \left(\frac{zye^{-x}}{1+y}\right)^{k}\sum_{l=0}^{\infty}\frac{(\eta')_l(\delta')_l}{(\xi')_{l}l!} \left(\frac{tye^{-x}}{1+y}\right)^{l}dxdy.
\end{align*}
\end{proof}

\begin{corollary}\label{cor3}
If $\delta=\delta'=1$ and $\xi=\xi'=1$,  then we get the result (22) of \cite{CP2017} as
\begin{align*}
\Phi_{\mu,\eta,\eta';\nu}\left(z,t,s,a\right)
&= \frac{\Gamma(\nu)}{\Gamma(s)\Gamma(\mu)\Gamma(\nu-\mu)}\\
&\times\int_{0}^{\infty}\int_{0}^{\infty}\frac{x^{s-1}e^{-ax}y^{\mu-1}}{(1+y)^{\nu}}\left(1-\frac{zye^{-x}}{1+y}\right)^{-\eta}\left(1-\frac{tye^{-x}}{1+y}\right)^{-\eta'}dxdy, \\
&\quad\quad\quad\quad\left(\Re(\nu)>\Re(\mu)>0;\min\left\{\Re(s), \Re(a)\right\}>0\right).
\end{align*}
\end{corollary}

\begin{theorem}\label{Th4}
The following summation formula hold true.
\begin{equation}\label{the4-eq1}
\sum_{r=0}^{\infty}\frac{(s)_r}{r!}\Phi_{\mu,\eta,\eta',\delta,\delta';\nu,\xi,\xi'}\left(z,t,s+r,a\right)x^{r}
=\Phi_{\mu,\eta,\eta',\delta,\delta';\nu,\xi,\xi'}\left(z,t,s,a-x\right), \quad \left(\left|x\right|<\left|a\right|; s\neq 1\right).
\end{equation}
\end{theorem}

\begin{proof}
Using \eqref{FEGHLZ}, we have
\begin{align*}
&\Phi_{\mu,\eta,\eta',\delta,\delta';\nu,\xi,\xi'}\left(z,t,s,a-x\right)\\
&=\sum_{k,l=0}^{\infty}\frac{(\mu)_{k+l}(\eta)_{k}(\eta')_{l}(\delta)_{k}(\delta')_{l}}{(\nu)_{k+l}(\xi)_{k}(\xi)_{l}}\frac{z^{k}t^{l}}{k!l!(k+l+a-x)^{s}},\\
&=\sum_{k,l=0}^{\infty}\frac{(\mu)_{k+l}(\eta)_{k}(\eta')_{l}(\delta)_{k}(\delta')_{l}}{(\nu)_{k+l}(\xi)_{k}(\xi)_{l}}\frac{z^{k}t^{l}}{k!l!(k+l+a)^{s}}\left(1-\frac{x}{k+l+a}\right)^{-s},\\
\text{using binomial series, we get}\\
&=\sum_{r=0}^{\infty}\frac{(s)_r}{s!}\left\{\sum_{k,l=0}^{\infty}\frac{(\mu)_{k+l}(\eta)_{k}(\eta')_{l}(\delta)_{k}(\delta')_{l}}{(\nu)_{k+l}(\xi)_{k}(\xi)_{l}}\frac{z^{k}t^{l}}{k!l!(k+l+a)^{s+r}}\right\}x^{r}.
\end{align*}

In view of definition \eqref{FEGHLZ}, we reach the required result.
\end{proof}

\section{A connection with generalized hypergeometric function}
In this section, we establish the connection between \eqref{FEGHLZ} and generalized hypergeometric function.
\begin{theorem}\label{Th5a}
For $a \neq \left\{-1,-2,\cdots \right\}$ and $z \neq 0$, the following explicit series representation holds true
\begin{equation}
\begin{array}{llll}\label{Th5}
\Phi_{\mu,\eta,\eta',\delta,\delta';\nu,\xi,\xi'}\left(z,t,s,a\right)
&=\sum_{k=0}^{\infty}\frac{(\mu)_{k}(\eta)_{k}(\delta)_{k}z^{k}}{(\nu)_{k}(\xi)_{k}(a+k)^{s}k!}\\
&\times F\left(\eta',\delta',1-\xi-k,-k;1-\eta-k,1-\delta-k,\xi';\frac{t}{z}\right),
\end{array}
\end{equation}
where $F(-)$ is the generalized hypergeometric function defined in \eqref{ghyp}.
\end{theorem}

\begin{proof}
Using \eqref{FEGHLZ}
and the identity \cite[page 56, Equation (1)]{ER}
\begin{equation*}
\sum_{k=0}^{\infty}\sum_{l=0}^{\infty}A\left(l,k\right)=\sum_{k=0}^{\infty}\sum_{l=0}^{k}A\left(l,k-l\right),
\end{equation*}
which implies that 
\begin{equation*}
\Phi_{\mu,\eta,\eta',\delta,\delta';\nu,\xi,\xi'}\left(z,t,s,a\right)=
\sum_{k=0}^{\infty}\sum_{l=0}^{k}\frac{(\mu)_{k}(\eta)_{k-l}(\eta')_{l}(\delta)_{k-l}(\delta')_{l}}{(\nu)_{k}(\xi)_{k-l}(\xi')_{l}(k-l)!l!}
\frac{z^{k-l}t^{l}}{(k+a)^{s}}.
\end{equation*}

Now, 
\begin{eqnarray*}
\left(k-l\right)!=\frac{(-1)^{l}k!}{(-k)_{l}}, 0 \leq l \leq k\\
\left(\eta\right)_{k-l}=\frac{(-1)^{l}(\eta)_{k}}{(1-\eta-k)_{l}}, 0 \leq l\leq k,
\end{eqnarray*}
we get,
\begin{align*}
&\Phi_{\mu,\eta,\eta',\delta,\delta';\nu,\xi,\xi'}\left(z,t,s,a\right)\\
&=\sum_{k=0}^{\infty}\sum_{l=0}^{k}\frac{(\mu)_{k}(\eta)_{k}(\eta')_{l}(\delta)_{k}(\delta')_{l}(1-\xi-k)_{l}(-k)_{l}z^{k-l}t^{l}}{(\nu)_{k}(1-\eta-k)_{l}(1-\delta-k)_{l}(\xi)_{k}(\xi')_{l}(k)!l!}\frac{1}{(k+a)^{s}}.
\end{align*}

Lastly, summing the $l$-series, we get the required result.
\end{proof}

\begin{corollary}\label{cor4}
If we set $\delta=\xi$ in Theorem \ref{Th5a}, then we get (28) of \cite{PD2012} as 
\begin{equation}
\Phi_{\mu,\eta,\eta',\xi,\delta';\nu,\xi,\xi'}\left(z,t,s,a\right)=\sum_{k=0}^{\infty}\frac{(\mu)_{k}(\eta)_{k}z^{k}}{(\nu)_{k}(k+a)^{s}k!}F\left(\eta',\delta',-k,1-\xi-k;\xi',1-\eta-k,1-\xi-k;\frac{t}{z}.\right)
\end{equation}
\end{corollary}

\begin{corollary}\label{cor5}
If we set $\nu=\eta,\delta=\xi$ and $\delta'=\xi'$ in Theorem \ref{Th5a}, then we get (29) of \cite{PD2012} as 
\begin{equation}
\Phi_{\mu,\eta,\eta',\xi,\xi';\eta,\xi,\xi'}\left(z,t,s,a\right)=\sum_{k=0}^{\infty}\frac{(\mu)_{k}z^{k}}{(a+k)^{s}k!}F\left(\eta',-k;1-\eta-k;\frac{t}{z}\right).
\end{equation}
\end{corollary}

 \section{Concluding Remarks}
An extension of a generalized Hurwitz-Lerch Zeta function is defined and some of its properties are studied in this paper. An integral representation is established and a relation with Appell’s type function is given. Finally, a connection with the hypergeometric function is also given. The results derived here are more general in nature by comparing the results of the papers \cite{CP2017, PD2012} which help to derive some interesting special cases and are mentioned in Remark \ref{rem1} and Corollaries \ref{cor1}, \ref{cor2}, \ref{cor3}, \ref{cor4}, and \ref{cor5}.



{\bf Declaration}:

{This article is a} corrected version of the preprint  {https://arxiv.org/abs/1706.03516}. 




\begin{thebibliography}{20}


\bibitem{Erdelyi}
          Erdelyi, A.
         \textit{Higher Transcendental Functions}; 1, McGraw Hill Book Co.: New York, NY, USA, 1953.
		
\bibitem{SK1985}			
				Srivastava, H.M.; Karlsson, P.W. 
				\textit{Multiple Gaussian Hypergeometric Series}; 
				Halsted Press (Ellis Horwood Limited): Chichester, UK; John Wiley and Sons: New York, NY., USA; Chichester, UK;   Brisbane, Australia; Toronto, ON, Canada, 1985.
				
				
\bibitem{MAK-GSA}
         Khan, M.A.; Abukhammash, G.S. 
         On a generalizations of Appell's functions of two variables.
				\textit{Pro Math.} \textbf{2002},~\textit{31}, 61--83.	
						
	\bibitem{HC2001}			
				Srivastava, H.M.; Choi, J. 
				\textit{Series Associated with the Zeta and Related Functions}; Kluwer, Acedemic Publishers: Dordrecht, The Netherlands; 
        Boston, MA, USA; London, UK, 2001.
				
\bibitem{HC2012}			
				Srivastava, H.M.; Choi, J.  
				\textit{Zeta and q-Zeta Functions and Associated Series and Integrals}; Elsevier Science: Amsterdam, The Netherlands; 
        London, UK; New York, NY, USA, 2012.

					
\bibitem{CZ2001}
          Chaudhry, M.A.; Zubair, S.M.  
					{\em On a Class of Incomplete Gamma Functions with Applications}; Chapman and Hall, (CRC Press Company):
					Boca Raton, FL, USA; London, UK; New York, NY, USA;  Washington, DC, USA, 2001.
					
\bibitem{CJS2008}			
					 Choi, J.; Jang, D.S.; Srivastava, H.M. 
					A generalization of the Hurwitz-Lerch Zeta function. {\em Integral Transf. Spec. Funct.} \textbf{2008}, \textit{19}, 65–79.
				
					
\bibitem{JPS2011}
					Jankov, D.; Pogany, T. K.; Saxena, R. K. 
					An extended general Hurwitz-Lerch Zeta function as a Mathieu $(a, \lambda)$-series. 
					\textit{Appl. Math. Lett.} \textbf{2011}, \textit{24} , 1473--1476.
	
\bibitem{LS2004}					
				 Lin, S.D.; Srivastava, H.M.  
				Some families of the Hurwitz-Lerch Zeta functions and associated fractional derivative and other integral representations. \textit{Appl. 	 Math. Comput.} \textbf{2004}, \textit{154}, 725--733.
				
\bibitem{HMS2014}
				Srivastava, H.M.  
				A new family of the $\lambda-$ generalized Hurwitz-Lerch Zeta functions with applications. 
				\textit{Appl.~Math. Inf. Sci.} \textbf{2014}, \textit{8}, 1485--1500.		
				
				
\bibitem{SJPS2011}	
				Srivastava, H.M.; Jankov, D.; Pogany, T.K.; Saxena, R.K. 
				Two-sided inequalities for the extended Hurwitz-Lerch Zeta function. 
				\textit{Comput. Math. Appl.} \textbf{2011}, \textit{62}, 516--522.
				
\bibitem{SLR2013} 
				Srivastava, H.M.; Luo, M.-J.; Raina, R.K.  
				New results involving a class of generalized Hurwitz-Lerch Zeta functions and their applications. 
				\textit{Turkish J. Anal. Number Theory} \textbf{2013}, \textit{1}, 26--35.
				
\bibitem{SSPS2011} 
				Srivastava, H.M.; Saxena, R.K.;  Pogany, T.K.; Saxena, R.
				Integral and computational representations of the extended Hurwitz-Lerch Zeta function. 
				\textit{Integral Transf. Spec. Funct.} \textbf{2011}, \textit{22}, 487--506.		
				
\bibitem{PD2012}
         Pathan, M.A.; Daman, O.
         Further generalization of Hurwitz Zeta function.
				\textit{Math. Sci. Res. J.} \textbf{2012},~\textit{16},~251--259.			

\bibitem{CP2017}
          Choi, J.; Parmar, R.
        An extension of the generalized Hurwitz-Lerch Zeta function of two variable.
        {\em Filomat} \textbf{2017},  \textit{31}, 91-96 
				
		
\bibitem{ER} Rainville, E.D. 
             \emph{Special Functions}; The Macmillan Company: New York, NY, USA, 1960.
				

												
				
\end{thebibliography}
\end{document}